\documentclass[11pt]{amsart}
\usepackage{amsmath}
\usepackage{amsthm}
\usepackage{amssymb}
\usepackage{amsfonts,mathrsfs}
\usepackage{stmaryrd}
\usepackage{amsxtra}  
\usepackage{epsfig}
\usepackage{verbatim}
\usepackage[all]{xy}

\usepackage{tikz}

\theoremstyle{plain}
\newtheorem{theorem}[equation]{Theorem}
\newtheorem{proposition}[equation]{Proposition}
\newtheorem{lemma}[equation]{Lemma}

\newtheorem{definition}[equation]{Definition}
\theoremstyle{remark}
\newtheorem{remark}[equation]{Remark}
\numberwithin{equation}{section}

\newcommand{\dbar}{\bar \partial}

\newcommand{\ch}{{\mathcal H}}

\newcommand{\cm}{{\mathcal M}}

\newcommand{\co}{{\mathcal O}}

\newcommand{\cv}{{\mathcal V}}

\newcommand{\sr}{{\mathscr R}}
\newcommand{\sS}{{\mathscr S}}

\newcommand{\C}{{\mathbb C}}

\newcommand{\Z}{{\mathbb Z}}

\begin{document}

\title[Regular vs. singular order]{Regular versus singular order of contact on pseudoconvex hypersurfaces}
\author{J.D. McNeal \& L. Mernik}

\begin{abstract}

The singular and regular type of a point on a real hypersurface $\ch$ in $\C^n$ are shown to agree when the regular type is strictly less than 4.
If $\ch$ is pseudoconvex, we show they agree when the regular type is 4. A non-pseudoconvex example is given 
where the regular type is 4 and the singular type is infinite.
\end{abstract}
 
\thanks{Research of the first author was partially supported by a National Science Foundation grant.}
\address{Department of Mathematics, \newline The Ohio State University, Columbus, Ohio, USA}
\email{mcneal@math.ohio-state.edu}
\address{Department of Mathematics, \newline The Ohio State University, Columbus, Ohio, USA}
\email{mernik.1@osu.edu}

\maketitle 


\section{Introduction}\label{S:intro}

Let $\ch\subset\C^n$ be a smooth, real hypersurface. D'Angelo introduced, \cite{DAngelo82}, a measurement of the holomorphic flatness
of $\ch$ at a point $p\in\ch$.  Let $\sS$ denote the set of parameterized non-constant holomorphic curves  $\gamma: V\longrightarrow\C^n$ with  $\gamma(0)=p$, where $V\subset\C$ is an unspecified neighborhood of $0$. Consider a local defining function $r$ for $\ch$: for some neighborhood $U$ of $p$, $\ch\cap U=\{z\in U: r(z) =0\}$ and $dr\neq 0$ in $U$.
Let $\nu(r\circ\gamma)=\nu(r\circ\gamma)(0)$ denote the order of vanishing of the real-valued function $r\circ\gamma$ at 0.  Also let $\nu(\gamma)$ denote the multiplicity of $\gamma$ at 0, i.e. the unique $M\in\Z^+$ such that 

$$\lim_{t\to 0}\frac{\gamma(t)-\gamma(0)}{t^M}\quad\text{exists and is  }\neq (0,\dots, 0).$$ 
Let $\sr\subset\sS$ denote the set of curves with $\nu(\gamma)=1$.

\medskip

\begin{definition} \label{D:type} For a given $\gamma\in \sS$, the order of contact of $\gamma$ with $\ch$ at $p$ is the (possibly infinite) quantity
$O(\gamma;p)= \frac{\nu(r\circ\gamma)}{\nu(\gamma)}$. The singular type of $p\in\ch$ is defined

\begin{equation}\label{D:sing_type}
\Delta_1(p)=\sup_{\gamma\in\sS} O(\gamma;p) .
\end{equation}
If $\Delta_1(p)<\infty$, $p$ is of finite singular type.

The regular type of $p\in\ch$ is defined

\begin{equation}\label{D:reg_type}
\Delta_1^{reg}(p)=\sup_{\gamma\in\sr}{\nu(r\circ\gamma)}.
\end{equation}
If $\Delta_1^{reg}(p)<\infty$, $p$ is of finite regular type.
\end{definition}
\medskip

The purpose of this paper is to prove 

\begin{theorem}\label{T:main}
Let $\ch\subset\C^n$ be a smooth real hypersurface and $p\in \ch$. 

\begin{itemize}
\item[(i)] If $\Delta_1^{reg}(p)\leq 3$, then $\Delta_1(p)\leq 3$ and $\Delta_1(p)= \Delta_1^{reg}(p)$.
\item[(ii)] If $\ch$ is pseudoconvex near $p$ and $\Delta_1^{reg}(p)= 4$, then  $\Delta_1(p)=4$.
\end{itemize}
\end{theorem}

The ratio defining $O(\gamma;p)$ implies this quantity is unchanged if the parameter variable $t$ is changed, e.g., to $t^k, k\in\Z^+$.
The subscript in $\Delta_1(p)$ indicates that orders of contact with 1-dimensional curves are considered. There are analogous higher-type conditions, denoted $\Delta_q(p)$ for $q=2,3,\dots , n$, discussed in \cite{DAngelo_SCVRHS}, but these conditions are more complicated to define. Finally, $\Delta_1(p)$ does not depend on the choice of defining function $r$; see \cite{DAngelo_SCVRHS}, Proposition 5 on page 114 for a proof.

The hypersurface $\ch$ can be viewed as the boundary $b\Omega$ of a smoothly bounded domain $\Omega\subset\C^n$. For function theory,  $\Delta_1(p)$ is important because it controls quantitative behavior of holomorphic functions on $\Omega$ near $p$; see \cite{DAngelo_SCVRHS} and its bibliography for results of this kind. Many analytic estimates in terms of  
$\Delta_1(p)$ have been obtained, but the story is far from complete. A remarkable result is obtained in \cite{Cat87}: on a smoothly bounded pseudoconvex domain, $\Delta_1(p)<\infty$ for $p\in b\Omega$ is a necessary and sufficient condition for a subelliptic estimate on the $\dbar$-Neumann problem to hold near $p$.

These connections with function theory motivate studying type but are unrelated to computing either $\Delta_1(p)$ or $\Delta_1^{reg}(p)$.
Both quantities depend solely on the germ at $p$ of the hypersurface $\{r=0\}$.

Since $\sr\subset\sS$,  $\Delta_1^{reg}(p)\leq \Delta_1(p)$ for any hypersurface. In general no other relationship between $\Delta_1^{reg}(p)$ and $\Delta_1(p)$ holds. For instance, consider the surface in $\C^3$ defined by $$r(z)=\text{Re }z_1 + \left|z_2^2-z_3^3\right|^2$$
near the origin. This example was considered in \cite{DAngelo79}. It is straightforward to check that $\Delta_1^{reg}(0)=6$. But the curve $t\longrightarrow (0, t^3, t^2)$ is contained in $\{r=0\}$, so $\Delta_1(0) =\infty$.
\smallskip

However there are hypotheses on $\ch$ that imply $\Delta_1^{reg}(p)= \Delta_1(p)$. When $\ch\subset\C^2$ this identity holds; see Theorem 9 on page 142 of \cite{DAngelo_SCVRHS}. In Proposition \ref{P:low_type} of Section \ref{S:jets}, we prove the measurements agree if $\Delta_1^{reg}(p)$ equals $2$ or $3$. When $\ch$ is pseudoconvex, the case $\Delta_1^{reg}(p)=2$ was known previously; see \cite{Kohn79}.
Results for more degenerate situations, i.e. when $\Delta_1^{reg}(p)> 3$, are unknown except for one class of hypersurfaces: it was shown in
 \cite{McNeal92} that $\Delta_1^{reg}(p)= \Delta_1(p)$, regardless of the size of $\Delta_1^{reg}(p)$, if $\ch$ locally bounds a {\it convex} domain. A geometric proof and generalization of this fact is given in \cite{BoasStraube92}.

In  \cite{DAngelo_SCVRHS},  page 148 in Proposition 3, it is stated that $\Delta_1^{reg}(p)= \Delta_1(p)$ if $\Delta_1^{reg}(p) =4$ on an arbitrary smooth real hypersurface. This turns out to be incorrect -- see Section \ref{S:example} for a counterexample. 
The main point of Theorem \ref{T:main} is that with the additional hypothesis of pseudoconvexity the conclusion is correct.
 
In view of  \cite{McNeal92} and Theorem \ref{T:main} (ii), there may be conditions -- intermediate between convexity and pseudoconvexity -- that imply $\Delta_1^{reg}(p)= \Delta_1(p)$ for points of type higher than 4. 
See Remark \ref{R:higher} for the obstructions to be controlled.

\section{Notation}\label{S:notation}

For curves $\gamma\in\sS$, 
let $t\in\Bbb{C}$ denote the parameter variable and write
$\gamma(t)=\left(\gamma^1(t), \dots , \gamma^n(t)\right)$ to indicate components. 

Derivatives will be denoted in several ways.
For functions defined on $\C^n$ or $\C$, subscripts will denote derivatives for small number of derivatives, e.g. $r_{z_j}$ or $\gamma^k_t$. For higher derivatives, the notation will  depend on which
function is differentiated.

For the function $r\circ\gamma$,

\begin{equation}\label{E:D_notation}
D^{a,b}[r\circ\gamma](t)=\frac{\partial^a}{\partial t^a}\frac{\partial^b}{\partial\bar t^b} (r\circ\gamma)(t)
\end{equation}
will distinguish $t$ and $\bar t$ derivatives. For the vector-valued function $\gamma$, the notation

\begin{equation*}
\left(\partial^k\gamma\right)(t)=\left(\frac{\partial^k}{\partial t^k} \gamma^1(t),\dots , \frac{\partial^k}{\partial t^k} \gamma^n(t)\right)\quad\text{and}\quad\left(\dbar^k \bar\gamma\right)(t)=\left(\frac{\partial^k}{\partial\bar t^k} \bar\gamma^1(t),\dots , \frac{\partial^k}{\partial\bar t^k} \bar\gamma^n(t)\right)
\end{equation*}
is used. 

When differentiating the defining function $r$, the following notation will be used. Let $a,b\in\Z^+$. Let  $X^k(t)=\left(X^k_1(t), \dots ,X^k_n(t)\right), 1\leq k\leq a$, and $Y^\ell (t)=\left(Y^\ell_1(t), \dots ,Y^\ell_n(t)\right), 1\leq\ell\leq b$, be given
vector-valued functions. Set

\begin{align}\label{E:nabla_notation}
\nabla^{a,b}[r]\left(X^1,\dots , X^a, Y^1,\dots , Y^b\right)(t)&= \notag\\
\sum_{j_1=1}^n\dots \sum_{j_a=1}^n \sum_{k_1=1}^n\dots &\sum_{k_b=1}^n r_{z_{j_1}\dots z_{j_a}\bar z_{k_1}\dots\bar z_{k_b}}\left(\gamma(t)\right) X^1_{j_1}(t)\dots X^a_{j_a}(t)\, Y^1_{k_1}(t)\dots Y^b_{k_b}(t)
\end{align}
as a definition of the left-hand side. The symbol $\nabla^{a,b}[r]$ is assigned no independent meaning; it will appear only via the action described by \eqref{E:nabla_notation}.

Standard multi-index notation is also used. For $\alpha=(\alpha_1,\dots ,\alpha_n)\in\left(\Z^+\right)^n$ and $h$ a function of $z\in\C^n$,

\begin{equation*}
\left(\frac\partial{\partial z}\right)^\alpha h(z)= \left(\frac\partial{\partial z_1}\right)^{\alpha_1}\cdots\left(\frac\partial{\partial z_n}\right)^{\alpha_n} h(z);
\end{equation*}
$\left(\frac\partial{\partial\bar z}\right)^\alpha$ is defined analogously.

When displaying derivatives, the underlying variables -- always $z\in\C^n$, $t\in\C$, or $\gamma(t)\in\C^n$ -- will be suppressed for notational economy. If variables do not appear in an equation, 
our implied meaning is that the equation holds functionally. To avoid confusion with this,
evaluation of derivative expressions at 0 will be explicitly indicated.


\section{An example}\label{S:example}

Consider the domain $\Omega=\left\{z\in\C^3: r(z) <0\right\}$ defined by

\begin{equation}\label{E:r}
r(z_1,z_2, z_3)= \text{Re}(z_1)+ |z_2|^2\, \text{Re}\left(z_2^2-z_3^3\right) + |z_3|^2\, \text{Re}\left(z_3^2\right) - \text{Re}\left(z_2^2\bar z_3\right).
\end{equation}
Note that $ p= (0,0,0)\in b\Omega$.
\medskip

{\bf Claim 1.} $\Delta_1(p)=\infty$.

\begin{proof} Consider the curve $\gamma(t)=\left(0, t^3, t^2\right)$, $t\in\C$. Then

\begin{align*}
(r\circ\gamma)(t)&= 0+\left|t^3\right|^2\, \text{Re}\left(t^6-t^6\right) +\left|t^2\right|^2\frac{t^4+\bar t^4}{2} -\frac{t^6\bar t^2+t^2\bar t^6}{2} \\
&= \frac12|t|^4\left(t^4+\bar t^4\right) - \frac12|t|^4\left(t^4+\bar t^4\right) =0.
\end{align*}
Thus (the image of) $\gamma$ lies in $b\Omega$, so $\Delta^1(p)=\infty$.
\end{proof}
\medskip

{\bf Claim 2.} $\Delta_1^{reg}(p)=4$.

\begin{proof} First consider the non-singular curve $\tilde\gamma(t)=(0,t,0)$. Since
$(r\circ\tilde\gamma)(t) =|t|^2\, \text{Re}\left(t^2\right)$, $\nu(r\circ\tilde\gamma) =4$ and so $\Delta_1^{reg}(p) \geq 4$.
We now show $\nu(r\circ\gamma)(0) \leq 4$ for any nonsingular curve $\gamma(t)=\left(\gamma^1(t), \gamma^2(t), \gamma^3(t)\right)$. Without loss of generality, assume that
$\gamma^1$ is identically 0 (see Subsection \ref{SS:psc_implies_equality}).

\noindent{\it Case 1:} $\nu(\gamma^2)=\nu(\gamma^3)=1$.

In this case, there are two degree 3 terms in $r\circ\gamma$, $t^2\bar t$ and $t\bar t^2$, coming from the last term on the right-hand side of \eqref{E:r}.
All other terms in $r\circ\gamma$ are of higher order. Thus $\nu(r\circ\gamma) =3$, so these curves are irrelevant for $\Delta_1^{reg}(0)$.

For the other cases, a slight rewrite of $r$ is useful:

\begin{equation}\label{E:r_1}
r(z)= \text{Re}(z_1) +\frac12 z_2^3\bar z_2 +\frac12 z_2\bar z_2^3 -\frac12|z_2|^2\left(z_3^3+\bar z_3^3\right) + |z_3|^2\, \text{Re}\left(z_3^2\right) - \text{Re}\left(z_2^2\bar z_3\right).
\end{equation}

\noindent{\it Case 2:} $\nu(\gamma^2)=1$ and $\nu(\gamma^3)>1$.

In this case, terms $t^3\bar t$ and $t\bar t^3$ appear in $r\circ\gamma$, because of the 2nd and 3rd terms on the right-hand side of \eqref{E:r_1}.
Terms coming from the 4th and 5th term in \eqref{E:r_1} will either be identically $0$ or have degree $\geq 8$, by assumption on $\nu(\gamma)$.
The last term on the RHS of \eqref{E:r_1} can produce degree $4$ terms,
but such terms would have {\it even} holomorphic degree in $t$ or {\it even} anti-holomorphic degree in $\bar t$, so cannot cancel the terms $t^3\bar t$ and $t\bar t^3$.

Thus, $D^{3,1} \left[r\circ \gamma\right](0)\neq 0$ for these $\gamma$, so $\nu(r\circ\gamma)\leq 4$.

\noindent{\it Case 3:} $\nu(\gamma^2) >1$ and $\nu(\gamma^3)=1$.

Similar reasoning applies to this case. Again there is a $t^3\bar t$ term in $r\circ\gamma$ (and also a $t\bar t^3$ term), this time arising from the 5th term on the right-hand side of \eqref{E:r_1}.
No other degree 4 terms are possible in this case, since we are assuming $\nu(\gamma_2) >1$. 

Thus, $D^{3,1} \left[r\circ \gamma\right](0)\neq 0$ for these curves as well, so $\nu(r\circ\gamma)\leq 4$.

\end{proof}


\section{Here come the warm jets}\label{S:jets}

If $\gamma\in\sS$, the statement $\nu(r\circ\gamma)(0)=T$ means 
\medskip

\begin{itemize}
\item[(i)]  for all $a,b\in\Z^+$ such that $0\leq a+b <T$, 
$$D^{a,b} [r\circ\gamma](0) =0,\quad\text{and}$$
\item[(ii)] $\exists\,$ an $a_0, b_0$ with $a_0+b_0=T$ such that 
$$D^{a_0,b_0} [r\circ\gamma](0)\neq 0.$$
\end{itemize}

To prove Theorem \ref{T:main}, we examine how derivatives $D^{a,b} [r\circ\gamma]$ for $\gamma$ with high multiplicity are related to
lower order derivatives of $r\circ\tilde\gamma$, for $\tilde\gamma\in\sr$ built from the jets of $\gamma$. The computations are local, near some $p\in b\Omega$; henceforth assume $p=0$.

\subsection{Automatic vanishing}\label{SS:auto}

Suppose $\gamma\in\sS$ and $\nu(\gamma)=M$; the case $M=1$ is included. Because $\gamma$ is holomorphic, the chain and product rules simplify when computing $D^{a,b} [r\circ\gamma]$. The early derivatives are

\begin{equation}\label{E:(1,0)}
D^{1,0} \left[r\circ\gamma\right] = \nabla^{1,0}[r]\left(\partial^1\gamma\right),
\end{equation}

\begin{equation}\label{E:(1,1)}
D^{1,1} \left[r\circ\gamma\right]= \nabla^{1,1}[r]\left(\partial^1\gamma, \dbar^1\bar\gamma\right),
\end{equation}

\begin{equation}\label{E:(2,1)}
D^{2,1} \left[r\circ\gamma\right]=  \nabla^{1,1}[r]\left(\partial^2\gamma,\dbar^1\bar\gamma  \right) + \nabla^{2,1}[r]\left(\partial^1\gamma,\partial^1\gamma, \dbar^1\bar\gamma \right)
\end{equation}

\begin{align}\label{E:(2,2)}
D^{2,2} \left[r\circ\gamma\right]&= \nabla^{1,1}[r]\left(\partial^2\gamma, \dbar^2\bar\gamma\right) + \nabla^{1,2}[r]\left( \partial^2\gamma,\dbar^1\bar\gamma,\dbar^1\bar\gamma\right) \notag \\
&+ \nabla^{2,1}[r]\left(\partial^1\gamma, \partial^1\gamma, \dbar^2\bar\gamma\right) \notag \\
&+ \nabla^{2,2}[r]\left( \partial^1\gamma, \partial^1\gamma,\dbar^1\bar\gamma,\dbar^1\bar\gamma\right)
\end{align}
and
\begin{align}\label{E:(3,1)}
D^{3,1} \left[r\circ\gamma\right]&= \nabla^{1,1}[r]\left(\partial^3\gamma,\dbar^1\bar\gamma\right) + \frac32\, \nabla^{2,1}[r]\left(\partial^2\gamma, \partial^1\gamma,\dbar^1\bar\gamma\right) \notag \\
&+ \nabla^{3,1}[r]\left(\partial^1\gamma, \partial^1\gamma, \partial^1\gamma,\dbar^1\bar\gamma\right).
\end{align}

Since $r$ is real-valued, $D^{b,a}[r\circ\gamma]=\overline{D^{a,b}[r\circ\gamma]}$. Thus \eqref{E:(1,0)}--\eqref{E:(3,1)}  contain all derivative information about $r\circ\gamma$ of total order $\leq 4$ except for the pure derivatives $D^{a,0}[r\circ\gamma]$ with $a=2,3,4$. These are easily handled. At the same time, we pick coordinates that identify the $T^{1,0}$ part of the tangent space to $\ch$ at 0. Choose holomorphic coordinates in a neighborhood of 0 such that

\begin{align}\label{E:good_coordinates}
r_{z_1}(0)&=1, \qquad r_{z_j}(0)=0\quad\text{  for } j=2,\dots n \notag\\
\left(\frac\partial{\partial z}\right)^\alpha r(0)&=\left(\frac\partial{\partial\bar z}\right)^\alpha r(0)=0\qquad\text{ for all }2\leq |\alpha|\leq 4M.
\end{align}
Such coordinates exist by elementary analysis of the Taylor expansion of $r$. For example, see the proof of Lemma 3.2.2 in \cite{Krantz_scv_book}.
From \eqref{E:(1,0)} and \eqref{E:good_coordinates} it follows that
\begin{align}\label{E:pure_terms}
D^{a,0}\left[r\circ\gamma\right](0)=\nabla^{1,0}[r](\partial^a\gamma)(0)
\end{align}
for $2\leq a\leq 4M$. 

Return to \eqref{E:(1,1)}--\eqref{E:(3,1)}, but ignore the first term on the right-hand side of \eqref{E:(2,2)}. 
At least one factor of $\partial^1\gamma$ or $\dbar^1\bar\gamma$ appears in each of the other terms. If $M\geq 2$, these vanish at $t=0$. Expressions of this kind proliferate as the number of derivatives increases.

To organize these, we make the following

\begin{definition}\label{D:mult_based_vanishing} Let $\gamma\in\sS$.
We say a term $A_m(t)=\nabla^{a,b}[r]\left(\partial^{j_1}\gamma,\dots,\partial^{j_a}\gamma, \dbar^{k_1}\bar\gamma,\dots, \dbar^{k_b}\bar\gamma\right)(t)\in TZ_\gamma$ if $\min\{j_1,...,j_a,k_1,...,k_b\}<M$.
For a general smooth function $G(t)=G(\gamma(t))$, write $G\in\cm\cv$ if  $G(t)=\sum (A_m(t))$ with each $A_m\in{\text TZ}_\gamma$.
If $G, H$ are two such functions, write $$G=H\quad\mod\left(\cm\cv\right)$$
if $G-H\in\cm\cv$. 
\end{definition}
The symbol $\cm\cv$ is short for multiplicity-vanishing. If $M\geq 2$, the equations \eqref{E:(1,1)}--\eqref{E:(3,1)} can be written
\begin{align*}
D^{1,1} \left[r\circ\gamma\right]&=0 \quad\mod\left(\cm\cv\right), \\
D^{2,1} \left[r\circ\gamma\right]&=0 \quad\mod\left(\cm\cv\right), \\
D^{2,2} \left[r\circ\gamma\right]&=\nabla^{1,1}[r]\left(\partial^2\gamma,\dbar^2 \bar\gamma\right) \quad\mod\left(\cm\cv\right),\\
D^{3,1}\left[r\circ\gamma\right]&=0 \quad\mod\left(\cm\cv\right).
\end{align*}

\subsubsection{Higher derivatives}\label{SSS:higher}
We now compute $D^{aM, bM}[r\circ\gamma]$ for $(a,b)= (1,1), (2,1), (2,2)$ and $(3,1)$. Computation and display of the results can be done mod $\cm\cv$, as only values at $t=0$ are needed.

For $D^{M,M}$,  \eqref{E:(1,1)} implies

\begin{equation}\label{E:(M,M)}
D^{M,M} \left[r\circ\gamma\right]=D^{M-1, M-1}D^{1,1} \left[r\circ\gamma\right]= \nabla^{1,1}[r]\left(\partial^M\gamma, \dbar^M\bar\gamma \right)\quad\mod\left(\cm\cv\right).
\end{equation}
The additional derivatives must all fall on $\gamma$ or $\bar\gamma\mod\left(\cm\cv\right)$. There is no mixing of $\frac{\partial}{\partial t}$ and $\frac{\partial}{\partial\bar t}$ derivatives since $\gamma$ is holomorphic.

For $D^{2M,M}[r\circ\gamma]$, \eqref{E:(2,1)} implies

\begin{align}\label{E:(2M,M)from(2,1)}
D^{2M,M} \left[r\circ\gamma\right]&=   D^{2M-2, M-1}\left\{D^{2,1} \left[r\circ\gamma\right]\right\}\notag \\
&= D^{2M-2, M-1}\left\{\nabla^{1,1}[r]\left(\partial^2\gamma,\dbar^1\bar\gamma  \right) + \nabla^{2,1}[r]\left(\partial^1\gamma,\partial^1\gamma, \dbar^1\bar\gamma \right)\right\}\notag \\
&=  D^{2M-2, M-1}\left\{ A+B\right\},
\end{align}
where the last equality defines $A$ and $B$.
Many terms arise when $D^{2M-2, M-1}\left\{ \,\,\right\}$ is expanded. Focus on the ``factors'' of $\partial^k\gamma$ and $\dbar^\ell\bar\gamma$ appearing in these terms. For example, $A$ contains two factors, one of $\partial^2\gamma$ and one of $\dbar^1\bar\gamma$, while $B$ contains three factors, two each of $\partial^1\gamma$ and one of $\dbar^1\bar\gamma$. Call these ``factors of $\gamma$ or $\bar\gamma$'' for short.
In order for a term {\it not} to belong to $\cm\cv$, each factor of $\gamma$ and $\bar\gamma$ that appears must be differentiated at least $M$ times.

For the ${\partial}/{\partial\bar t}$ derivatives, \eqref{E:(2M,M)from(2,1)} shows all $M-1$ derivatives must fall on the single factor $\dbar^1\bar\gamma\mod\left(\cm\cv\right)$ in both $A$ and $B$. However the $\frac{\partial}{\partial t}$ derivatives fall on multiple factors of $\partial^k\gamma$ appearing in $B$ and (eventually) in $A$. For $B$, no $\frac{\partial}{\partial t}$ derivative may fall on $r\mod\left(\cm\cv\right)$: each of the two factors $\partial^1\gamma$ must be differentiated $M-1$ more times and there are only $2M-2$ total derivatives. For $A$, one $\frac{\partial}{\partial t}$ derivative may fall on $r$, but then the remaining derivatives must all be distributed between the newly created factor of 
$\partial^1\gamma$ and the originally present $\partial^k\gamma$ factor, resulting in dual factors of $\partial^M\gamma$ at the end. Thus

\begin{align}\label{E:(2M,M)}
D^{2M,M} \left[r\circ\gamma\right]&= \nabla^{1,1}[r]\left(\partial^{2M}\gamma,\dbar^M\bar\gamma\right)\notag\\
&+E_0\, \nabla^{2,1}[r]\left(\partial^M\gamma, \partial^M\gamma, \dbar^M\bar\gamma\right)
\quad\mod\left(\cm\cv\right).
\end{align}
$E_0$ is a combinatorial constant arising from the product rule; it is computed in the next section.
Notice the higher-than-multiplicity derivatives $\partial^{2M}\gamma$ in \eqref{E:(2M,M)}. This factor may or may not be 0 for a given $\gamma\in\sS$. 

Computing the $D^{2M, 2M}$ derivative is very similar, as the $\frac{\partial}{\partial t}$ and $\frac{\partial}{\partial\bar t}$ derivatives do not intermix . Calculation directly from \eqref{E:(2,2)} yields 

\begin{align}\label{E:(2M,2M)}
D^{2M,2M} \left[r\circ\gamma\right]&=\nabla^{1,1}[r]\left(\partial^{2M}\gamma,\dbar^{2M}\bar\gamma\right) \\ \nonumber
&+ F_0\, \nabla^{2,1}[r]\left(\partial^M\gamma, \partial^M\gamma, \dbar^{2M}\bar\gamma\right)\\ \nonumber
&+ F_1\, \nabla^{1,2}[r]\left(\partial^{2M}\gamma,\dbar^M \bar\gamma,\dbar^M\bar\gamma\right)\\ \nonumber
&+ F_2 \, \nabla^{2,2}[r]\left(\partial^M \gamma,\partial^M\gamma,\dbar^M \bar\gamma,\dbar^M \bar\gamma\right)\quad\mod\left(\cm\cv\right).
\end{align}
The constants $F_0, F_1, F_2\in\Z^+$ are also computed in subsection \ref{SSS:combo}.

For the final relevant derivative, $D^{3M, M}$, a variation occurs. The surplus $\frac{\partial}{\partial t}$ derivatives allows several combinations of higher-than-multiplicity derivatives of $\gamma$ to occur
inside $\nabla^{2,1}[r]$. The appearance of these (potentially) non-vanishing terms is the reason $\Delta_1(p)=\Delta_1^{reg}(p)$ does not hold without additional hypothesis.

Starting from \eqref{E:(3,1)} it follows

\begin{align}\label{E:(3M,M)}
D^{3M,M} \left[r\circ\gamma\right]&=\nabla^{1,1}[r]\left(\partial^{3M}\gamma,\dbar^{M}\bar\gamma\right) \\ \nonumber
&+G_0\, \nabla^{2,1}[r]\left(\partial^{2M}\gamma,\partial^M\gamma,\dbar^{M}\bar\gamma\right)\\ \nonumber
&+ \sum_{i=1}^{M-1} G_i\, \nabla^{2,1}[r]\left(\partial^{M+i}\gamma,\partial^{2M-i}\gamma,\dbar^{M}\bar\gamma\right)\\ \nonumber
&+ G_M\, \nabla^{3,1}[r]\left(\partial^M \gamma,\partial^M \gamma,\partial^M\gamma,\dbar^{M}\bar\gamma\right)\quad\mod\left(\cm\cv\right),
\end{align}
for constants $G_0,\dots G_M\in\Z^+$.

\subsubsection{Combinatorial constants}\label{SSS:combo}
The general forms of $D^{aM,bM}$ were derived starting from the corresponding lower-order derivatives
 \eqref{E:(2,1)}-\eqref{E:(3,1)}. However the constants in these equations are most easily computed starting from $r\circ\gamma$ itself.
 
For a term in a derivative expression to not belong to $\cm\cv$, each factor of $\gamma$ and $\bar\gamma$ present must be differentiated at least $M$ times. Consider \eqref{E:(2M,M)}:

\begin{align*}
D^{2M,M} \left[r\circ\gamma\right]&= \nabla^{1,1}[r]\left(\partial^{2M}\gamma,\dbar^M\bar\gamma\right)+E_0\, \nabla^{2,1}[r]\left(\partial^M\gamma, \partial^M\gamma, \dbar^M\bar\gamma\right)
\quad\mod\left(\cm\cv\right) \\ &=\tilde A+E_0\, \tilde B.
\end{align*} 
The factor $\partial^{2M}\gamma$ occurs when all the $\frac{\partial}{\partial t}$ derivatives of $r\circ\gamma$ fall on a single $\gamma$. This happens in only one way, so the constant in front of $\tilde A$ is 1. The two factors $\partial^M\gamma$ in $\tilde B$ occur when the $\frac{\partial}{\partial t}$ derivatives of $r\circ\gamma$ are distributed between two factors of $\gamma$. Distinguish these two factor of $\gamma$ with different symbols: $\gamma_1$ and $\gamma_2$. There is more than one way to reach a term of the form $\partial^M\gamma_1\cdot\partial^M\gamma_2$,  as the order of differentiation does not matter.
 For instance, first $\gamma_1$ can be differentiated $M$ times, followed by differentiating $\gamma_2$ $M$ times, or $\gamma_2$ could first be differentiated $M$ times, followed by $M$ derivatives of $\gamma_1$, as well as all intermediate options.  The term $\tilde B$ contains the sum of such terms, hence $E_0\neq 1$.
 
 Counting how many ways we can differentiate both $\gamma_1$ and $\gamma_2$ exactly $M$ times,  using $2M$ total derivatives, is equivalent to counting all sequences of length $2M$ with $M$ $1$'s and $M$ $2$'s. For instance, the sequence $\underbrace{1\dots 1}_M\underbrace{2\dots 2}_M$ corresponds to first differentiating $\gamma_1$ $M$ times followed by differentiating $\gamma_2$ $M$ times.  The number of all such sequences is $\frac{(2M)!}{M!M!}$. This over-counts $E_0$ by $2!$,  as all  sums in $\nabla^{2,1}$ runs from $1$ to $n$ so the roles of $1$ and $2$ can be exchange . Thus
 \begin{equation}\label{E:E_0}
E_0=\frac12 \frac{(2M)!}{M!M!}. 
\end{equation}

The same approach computes the constants in  \eqref{E:(2M,2M)}-\eqref{E:(3M,M)}, yielding

\begin{align}\label{E:constants}
F_0=\frac12 \frac{(2M)!}{M!M!},\quad F_1&=\frac12 \frac{(2M)!}{M!M!}, \quad  F_2=\frac12\frac{(2M)!}{M!M!}\frac12\frac{(2M)!}{M!M!},\notag \\
G_0=\frac12 \frac{(3M)!}{(2M)!M!}, \quad G_i&=\frac12 \frac{(3M)!}{(M+i)!(2M-i)!}, \text{ and } G_M=\frac16\frac{(3M)!}{M!M!M!}.
\end{align}

\subsection{Connected derivatives}\label{SS:connected}

We now parameterize $\gamma$. Let

\begin{equation}\label{E:parameter_gamma}
\gamma (t)=\left(t^M\, \sum_{i=0}^\infty c^1_i t^i, \dots, t^M\, \sum_{i=0}^\infty c^n_i t^i\right)
\end{equation}
for constants $c_i^q\in\C$, $q=1,\dots, n$ and $i\in\{0,1,\dots\}$. Since $\nu(\gamma)=M$, the vector $\left\langle c^1_0,\dots , c^n_0\right\rangle\neq \langle 0,\dots ,0\rangle$. 

The subscript notation for derivatives is used in the computations below. This notation makes the interaction of the $c_j^k$ in \eqref{E:parameter_gamma} -- which correspond to components of $\partial^j\gamma$ -- with derivatives of
$r$ easier to observe. The formulas \eqref{E:(M,M)}-\eqref{E:(3M,M)} applied to $\gamma$ of the form \eqref{E:parameter_gamma} result in cancellation of various factorials, as seen below.

When $\Delta_1^{reg}(p) < 4$, the computations in \ref{SSS:higher} show $\Delta_1(p)=\Delta_1^{reg}(p)$ for a general hypersurface. 

\begin{proposition}\label{P:low_type}
Let $\ch\subset\C^n$ be a smooth real hypersurface. Suppose $p\in \ch$ and $\Delta_1^{reg}(p)\leq 3$. 
Then $\Delta_1(p)\leq 3$ and $\Delta_1(p)=\Delta_1^{reg}(p)$.
\end{proposition}

\begin{proof}
The cases $\Delta_1^{reg}(p)=2$ and $\Delta_1^{reg}(p)= 3$ will be handled separately.
 
First suppose that $\Delta_1(0) >2$. Then there is a curve $\gamma$ of the form \eqref{E:parameter_gamma} satisfying $\nu(r\circ\gamma)(0) >2M$:

\begin{equation}\label{E:type2_vanishing}
D^{a,b}\left[r\circ\gamma\right](0)=0 \qquad\text{ for all }0\leq a+b\leq 2M.
\end{equation}
Consider the non-singular curve $\hat\gamma(t)= \left(c_0^1 t, \dots, c_0^n t\right)$. It follows from \eqref{E:pure_terms} that 

\begin{equation}\label{E:pure_matching}
D^{M,0}\left[r\circ\gamma\right](0) =M! D^{1,0}\left[r\circ\hat\gamma\right](0).
\end{equation}
The left-hand side vanishes by \eqref{E:type2_vanishing}, so \eqref{E:pure_matching} says $D^{1,0}\left[r\circ\hat\gamma\right]=0$. $D^{0,1}\left[r\circ\hat\gamma\right]=0$ follows by conjugation.
Geometrically this says $\left\langle c^1_0,\dots , c^n_0\right\rangle$ is a complex tangent vector to $\ch$ at 0.
In the same way \eqref{E:pure_terms} shows that $D^{2,0}\left[r\circ\hat\gamma\right]=0=D^{0,2}\left[r\circ\hat\gamma\right]$.

It follows from \eqref{E:(M,M)} that

\begin{equation*}
D^{M,M}\left[r\circ\gamma\right](0) = M!M!\sum_{k,j=1}^n r_{\bar z_k z_j}\, \bar{c^k_0} c^j_0.
\end{equation*}
\eqref{E:type2_vanishing} says the left-hand side vanishes.
On the other hand, \eqref{E:(1,1)} implies

\begin{equation*}
D^{1,1} \left[r\circ\hat{\gamma}(0)\right] = \sum_{k,j=1}^n r_{\bar z_k z_j}\, \bar{c^k_0} c^j_0.
\end{equation*}
Equating right-hand sides yields $D^{1,1}\left[r\circ\hat\gamma(0)\right]=0$. Thus $\hat\gamma\in\sr$ satisfies $D^{a,b}\left[r\circ\hat\gamma(0)\right]=0$ for all $0\leq a+b\leq 2$, i.e.,
$\Delta_1^{reg}(0) >2$.

Turn to type 3 points. Assume $\gamma$, of the form \eqref{E:parameter_gamma}, satisfies $\nu(r\circ\gamma)(0) >3M$. In this case the complex line $\hat\gamma$ above is not  the right element of $\sr$.  The first term in  \eqref{E:(2M,M)} dictates a modification.  We claim that 

$$\tilde\gamma(t)=(c_0^1t+ c_M^1t^2,\dots, c_0^nt+ c_M^nt^2)$$ 
satisfies $\nu(r\circ\tilde\gamma)(0)>3$. Note $\tilde\gamma\in\sr$. 

The computations that gave \eqref{E:pure_matching} show 

$$D^{1,0}[r\circ\tilde\gamma]=D^{2,0}[r\circ\tilde\gamma]=D^{3,0}[r\circ\tilde\gamma]=0.$$
As before, $D^{0, b} [r\circ\tilde\gamma]=0$ for $b=1,2,3$ because $r$ is real-valued. Additionally, the computation giving $D^{1,1}[r\circ\hat\gamma]=0$
shows $D^{1,1}[r\circ\tilde\gamma]=0$.

It remains to consider $D^{2,1}[r\circ\tilde\gamma]$. 
Applying  \eqref{E:(2M,M)} to $\gamma$ of the form \eqref{E:parameter_gamma} gives

$$D^{2M,M}[r\circ\gamma]=(2M)!M!\sum_{k,j=1}^n r_{\bar z_k z_j} \bar c_0^k c_M^j + \frac12 (2M)!M!\sum_{\ell,k,j=1}^n r_{z_\ell\bar z_k z_j} c_0^\ell\bar c_0^k c_0^j \quad\mod\left(\cm\cv\right).$$
However \eqref{E:(2,1)}, applied to $\tilde\gamma$, yields

 $$D^{2,1}[r\circ\tilde\gamma]=\sum_{k,j=1}^n r_{\bar z_kz_j}\bar c_0^k (2c_M^j)+ \sum_{\ell,k,j=1}^n r_{z_\ell\bar z_k z_j} c_0^\ell\bar c_0^k c_0^j.$$
Simple comparison shows $D^{2,1}[r\circ\tilde\gamma](0)=\frac2{(2M)!M!} D^{2M,M}[r\circ\gamma](0)$. Since $\nu(r\circ\gamma)(0)>3M$ forces $D^{2M,M}[r\circ\gamma](0)=0$, it follows that $D^{2,1}[r\circ\tilde\gamma](0)=0$ as well.

Thus $D^{a,b}[r\circ\tilde\gamma](0)=0$ for $0\leq a+b\leq 3$, i.e., $\Delta_1^{reg}(0)>3$.

\end{proof}

\begin{remark} If $\ch$ is pseudoconvex (see \eqref{E:psc}), Proposition 5.16 of \cite{Kohn79} establishes the following equivalence:  $\Delta_1(p) =2\Leftrightarrow\Delta_1^{reg}(p)=2 \Leftrightarrow$ the Levi form of $r$ at $p$  is definite. 
Our proof of the first equivalence does not require pseudoconvexity and is elementary. The proof in  \cite{Kohn79} is more complicated, but also connects the equivalence to subelliptic multipliers at $p$.
\end{remark}

The argument in Proposition \ref{P:low_type} does not extend to type 4 points, as the example in Section \ref{S:example} shows. We identify where the breakdown occurs.
Rewriting \eqref{E:(3M,M)}, for $\gamma$ of the form \eqref{E:parameter_gamma}, yields 

\begin{align}\label{E:(3M,M)_para}
D^{3M,M}[r\circ\gamma](0)=(3M)!M!\sum_{j,k=1}^n r_{z_j\overline{z_k}}(0)\, c^j_{2M}\,\overline{c_0^k} +\frac12(3M)!M!\sum_{j,k,\ell=1}^n r_{z_jz_\ell\overline{z_k}}(0)\, c^j_M\, c^\ell_0\,\overline{c^k_0} + 
\\ \nonumber
\underbrace{\frac12(3M)!M!\sum_{j,k,\ell=1}^n r_{z_jz_\ell\overline{z_k}}(0)\,\overline{c^k_0}\left(\sum_{i=1}^{M-1} c^j_i\,c^\ell_{M-i}\right)}_{B(r;\gamma)}+\frac16(3M)!M!\sum_{j,k,\ell,m=1}^n r_{z_jz_\ell z_m\overline{z_k}}(0)\, c^j_0 c^\ell_0 c^m_0\overline{c^k_0}
\end{align}
Assume that $\nu(r\circ\gamma) >4$.
The terms labeled $B(r;\gamma)$ obstruct construction of a non-singular curve with order of contact $>4$ by the matching-of-derivatives argument used in Proposition \ref{P:low_type}.
Note the presence of mixed coefficients of the components of $\gamma^j$, $\gamma^k$, and $\gamma^\ell$ (for various $j,k,\ell$) appearing in $B(r;\gamma)$.

This turns out to be the only obstruction:

\begin{lemma}\label{L:B}
Let $\gamma\in\sS$, $\nu(\gamma)=M$, and suppose $\nu(r\circ\gamma)>4M$. 
If $B(r;\gamma)=0$ in \eqref{E:(3M,M)_para}, then there exists $\zeta\in\sr$ satisfying $\nu(r\circ\zeta)>4$.
\end{lemma}

\begin{proof}
For $\gamma$ of the form \eqref{E:parameter_gamma}, define

\begin{equation*}
\zeta(t)=\left(c^1_0 t+c^1_M\, t^2+ c^1_{2M}\,t^3,\dots,c^n_0 t+c^n_M\,t^2+ c^n_{2M}\,t^3\right). 
\end{equation*}
Since $\left\langle c^1_0,\dots , c^n_0\right\rangle\neq \langle 0,\dots ,0\rangle$,  $\zeta\in\sr$. 

The same computations used in Proposition \ref{P:low_type} show
\begin{equation*}
\frac{ (aM)!(bM)!}{a!b!}D^{a,b}[r\circ \zeta](0)=D^{aM, b M} [ r\circ \gamma](0)=0
\end{equation*} for $0\leq a+b\leq3$.
A computation using \eqref{E:(2,2)} and \eqref{E:(2M,2M)} shows that 

$$\frac{(2M)!(2M)!}{2!2!}D^{2,2}[r\circ\zeta](0)=D^{2M,2M}[r\circ\gamma](0)=0 .$$
This computation is essentially the same as the one showing $M!M!D^{1,1}[r\circ\hat\gamma](0)=D^{M,M}[r\circ\gamma](0)$  in Proposition \ref{P:low_type}.
Note no extra assumptions were needed to match $D^{aM,bM}$ derivatives of $r\circ\gamma$ with $D^{a,b}$ derivatives of $r\circ\zeta$ in these cases.

The remaining derivatives to consider are $D^{3,1}[r\circ\zeta](0)$ and $D^{3M,M}[r\circ\gamma](0)$. Focusing on the terms $B(r; \gamma)$ in \eqref{E:(3M,M)}, note the appearance of derivatives of order not a multiple of the multiplicity $M$. These are the problematic terms.

Use \eqref{E:(3,1)} to compute $D^{3,1}[r\circ\zeta](0)$:

\begin{align*}
 D^{3,1}[r\circ\zeta](0)= &\sum_{k,j} r_{\bar z_k z_j}(0) (3!c_{2M}^j) \bar c_0^k + \frac32\sum_{\ell,k,j=1}^n r_{z_\ell\bar z_k z_j}(0) c_0^\ell \bar c_0^k (2c_M^j) \\
&+ \sum_{m,\ell,k,j=1}^n r_{z_j z_\ell z_m\bar z_k}(0)\, c_0^m c_0^\ell \bar c_0^k c_0^j .
\end{align*}
Comparing this to \eqref{E:(3M,M)_para}, it follows that 

$$\frac{(3M)!M!}{3!}D^{3,1}[r\circ\zeta](0)=D^{3M,M}[r\circ\gamma](0)-B(r;\gamma).$$ 

Therefore if $B(r;\gamma)=0$, $D^{3,1}[r\circ\zeta](0)=0$, since $\nu(r\circ\gamma)>4M$. Thus $\nu(r\circ\zeta) >4$, completing the proof.
\end{proof} 

\begin{remark}\label{R:higher}
     When seeking an extension of Theorem \ref{T:main} for $\Delta_1^{reg}(p) >4$, higher order analogs of $B(r;\gamma)$ arise. For instance in the type 6 case, analogs of  $B(r;\gamma)$ 
    corresponding to derivatives $D^{4,1}, D^{5,1}, D^{4,2}, D^{3,2}$, and $D^{3,3}$ occur. Label these  $B_{4,1}, B_{5,1}, B_{4,2}, B_{3,2}$, and $B_{3,3}$.
    In each case, $B_{a,b}$ represents the difference between $D^{aM,bM}$ derivatives of $r\circ\gamma$, where $\nu(\gamma)=M$, and $D^{a,b}[r\circ \zeta]$
    for a natural non-singular curve $\zeta$ constructed from $\gamma$. If all $B_{a,b}=0$, a straightforward modification of the proof of Lemma \ref{L:B} yields the correspondingly modified conclusion.
\end{remark}


\section{Proof of Theorem \ref{T:main}}\label{S:proof}

\subsection{Pseudoconvexity}\label{SS:psc}
Let $\ch\subset \C^n$ be a smooth real hypersurface, $q\in\ch$, and let $\rho$ be a defining function for $\ch$ near $q$.
The complex tangent space, $T^{1,0}(\ch; q)$,  to $\ch$ at $q$ is defined to be
all $\xi=(\xi_1,\dots ,\xi_n)\in\C^n$ satisfying $\sum_{j=1}^n\frac{\partial \rho}{\partial z_j}(q)\xi_j=0$. The vector space $T^{1,0}(\ch; q)$
is independent of the choice of defining function.

The Levi form associated to $\ch$ is a certain quadratic form on $T^{1,0}(\ch; q)$; pseudoconvexity is the condition that the Levi form is semi-definite. 
There is an arbitrary choice of sign for this semi-definiteness, which we now fix. Say $\ch$ is {\it pseudoconvex} if there exists a defining function
$r$ of $\ch$ such that

\begin{equation}\label{E:psc}
\sum_{j,k=1}^n\frac{\partial^2 r}{\partial z_j\partial\bar z_k}(q)\xi_j\bar\xi_k\geq 0\qquad\text{ for all } \xi\in T^{1,0}(\ch; q),
\end{equation}
for all $q\in\ch$. $\ch$ is pseudoconvex near $p\in \ch$ if \eqref{E:psc} holds for all $q\in U\cap\ch$,
$U$ a neighborhood of $p$.

Pseudoconvexity implies the existence of good holomorphic coordinates near $p$:

\begin{proposition}\label{P:psc_main} Let $\ch\subset\C^n$ be a pseudoconvex smooth real hypersurface and $p\in\ch$. 

There exists a defining function $r$ for $\ch$,  coordinates $(z_1, z_2,\dots , z_n)=\left(z_1, z'\right)$, and constants $\kappa_2,\dots ,\kappa_n\geq 0$, $\lambda_{jk\ell}\in\C$, such that $p=(0,\dots ,0)$ and

\begin{align}\label{E:normal_r}
r(z)= 2\text{Re }z_1 +\sum_{j=2}^{n} \kappa_{j} \left|z_j\right|^2 +\text{Re }\left(\sum_{j,k,\ell=1}^{n} \lambda_{jk\ell}\, z_j z_j \bar z_\ell \right)
 +\co\left(\left|\text{Im }z_1\right|, \left|z'\right|^4\right)
\end{align}

\end{proposition}

\begin{proof} This is generally well-known. Start with the coordinates \eqref{E:good_coordinates} and an arbitrary defining function $\rho$ for $\ch$. Apply the implicit function theorem to express $\{\rho=0\}$ as the set 
$\left\{2\text{Re }z_1=F\left(z',\text{Im }z_1\right)\right\}$ locally near $0$. A $\C$-linear rotation of coordinates allows us to assume that $F\left(z',\text{Im }z_1\right)$
vanishes to at least order 2 at 0.
Define $r(z)= 2\text{Re }z_1 -F\left(z',\text{Im }z_1\right)$.

Taylor's theorem then gives

\begin{align*}
r(z)&= 2\text{Re }z_1 +\sum_{j,k=2}^{n} \kappa_{jk}\, z_j\bar z_k +\sum_{j,k,\ell=1}^{n} \lambda_{jk\ell}\, z_j z_j \bar z_\ell 
+ \sum_{j,k,\ell=1}^{n} \mu_{jk\ell}\, z_j \bar z_k \bar z_\ell +\co\left(\left|\text{Im }z_1\right|, \left|z'\right|^4\right)\notag \\
&= 2\text{Re }z_1 + P_2\left(z',{\bar z}'\right) + P_3(z,\bar z) + \co\left(\left|\text{Im }z_1\right|, \left|z'\right|^4\right).
\end{align*}

The fact that $r$ is real-valued implies that $\overline{\mu_{jk\ell}}={\lambda_{jk\ell}}$; thus $P_3(z,\bar z)$ is of the form claimed by \eqref{E:normal_r}.
The matrix $\left(\kappa_{jk}\right)$ coming from $P_2\left(z',{\bar z}'\right)$ is Hermitian, so can be diagonalized; let $A$ be a matrix such that $A^*\cdot\left(\kappa_{jk}\right)\cdot A$ is diagonal. Making the linear change of coordinates $z' =Az'$ then puts $ P_2\left(z',{\bar z}'\right)$ in the form claimed by \eqref{E:normal_r} -- note that pseudoconvexity implies the eigenvalues of the matrix $\left(\kappa_{jk}\right)$ are non-negative -- and does not disturb the form of the rest of the right-hand side of \eqref{E:normal_r}. 

\end{proof}

Our secondary use of \eqref{E:psc} is in conjunction with curves  tangent to $\ch$. 
The fact we need is proved as Proposition 2 on page 138 of \cite{DAngelo_SCVRHS}:

\begin{proposition}\label{P:psc_dangelo}
Let $\ch$ be a pseudoconvex smooth real hypersurface, $0\in\ch$, and $r$ a smooth defining function for $\ch$. Suppose that $\gamma\in\sS$
satisfies 
\begin{align*}
1< \nu(r\circ\gamma)(0) &=T  \\
D^{a,0} [r\circ\gamma](0) &=0\qquad\forall\,\, a\leq T.
\end{align*}

Then

\begin{itemize}
\item[(i)] $T=2K$, for some $K\in\Z^+$, and
\item[(ii)] in the Taylor expansion of $r\circ\gamma(t)$, the coefficient of $|t|^{2K}$ is positive.
\end{itemize}
\end{proposition}

Proposition \ref{P:psc_dangelo} is an extension of the following fact. If $u(t)$ is a smooth, subharmonic function with no pure terms in its Taylor expansion and $u$ vanishes to finite order at 0, then the order of vanishing is even ($2K$ above) and
the coefficient of $|t|^{2K}$ in the Taylor expansion of $u$ is positive.

\subsection{Type 4: pseudoconvexity implies equality}\label{SS:psc_implies_equality}

We now prove Theorem \ref{T:main}. The coordinates given by Proposition \ref{P:psc_main} are used throughout, in particular $p$ is the origin.

Because of Proposition \ref{P:low_type}, it suffices to consider the case $\Delta_1^{reg}(0) =4$. Suppose that $\Delta_1(0) > 4$. Then there is a curve $\gamma$ of the
form \eqref{E:parameter_gamma},

\begin{equation}\label{E:parameter_gamma1}
\gamma (t)=\left(t^M\, \sum_{i=0}^\infty c^1_i t^i, \dots, t^M\, \sum_{i=0}^\infty c^n_i t^i\right),
\end{equation}
with $\left\langle c^1_0,\dots , c^n_0\right\rangle\neq \langle 0,\dots ,0\rangle$ and $M\geq 2$, satisfying $\nu(r\circ\gamma) > 4M$. We will show there is a curve $\zeta\in\sr$ with
$\nu(r\circ\zeta) >4$, contradicting the assumption that $\Delta_1^{reg}(0) =4$.

We first reduce the complexity of $\gamma$. Instead of a general curve satisfying $\nu(r\circ\gamma) > 4M$, we may consider one with {\it maximal} vanishing order. It  follows from the proof of Theorem 1 on page 127-128 of \cite{DAngelo_SCVRHS} that there is such a curve, $\tilde\gamma=\left(\tilde\gamma^1,\dots, \tilde\gamma^n\right) $, lying in the complex tangent space  $T^{1,0}(\ch; 0)$. In the coordinates \eqref{E:good_coordinates}, $T^{1,0}(\ch; 0)=\{z_1=0\}$, so this implies $\tilde\gamma^1\equiv 0$. Thus, without loss of generality, we assume that $\gamma$ is of the form

\begin{equation}\label{E:parameter_gamma2}
\gamma (t)=\left(0, t^M\, \sum_{i=0}^\infty c^2_i t^i, \dots, t^M\, \sum_{i=0}^\infty c^n_i t^i\right),
\end{equation}
at the onset.

Next we infer information about the form of $r$ in \eqref{E:normal_r}:  

\begin{lemma}\label{L:1} If $\lambda_{j_0k_0\ell_0} \neq 0$ in \eqref{E:normal_r}, for some $j_0, k_0, \ell_0$ -- not necessarily distinct -- then at least one of the constants $\kappa_{j_0}$, $\kappa_{k_0}$, or $\kappa_{\ell_0}$ must be strictly positive. 

(They may all be strictly positive.)
\end{lemma}

\begin{proof}Consider the curve $\eta: t\longrightarrow\left(0, \eta^2(t),\dots , \eta^n(t)\right)$ with 

\begin{align*}
\eta^i(t)= \begin{cases}
    \delta_i t, & i\in\{j_0, k_0, \ell_0\}\\
    0, & i\notin\{j_0, k_0, \ell_0\}
  \end{cases}\;\;,
\end{align*}
for constants $\delta_i\neq 0$ to be chosen. Choose $\delta_i$ such that $r\circ\eta(t)$ has a non-zero $t^2\bar t$ term. Such $\delta_i$ can always be chosen as terms in $P_3(z',\bar z')$ involving only $z_{j_0},z_{k_0},$ or $z_{\ell_0}$  form a non-zero polynomial. If $P_2\left(z',{\bar z}'\right) $ contains no terms involving only  $z_{j_0}$, $z_{k_0}$, and $z_{\ell_0}$, $\nu(r\circ\eta)(0)=3$, contradicting Proposition \ref{P:psc_dangelo} (i). 
\end{proof}

However, a non-vanishing quadratic term in \eqref{E:normal_r} forces higher order vanishing of that component of $\gamma$:

\begin{lemma}\label{L:2}
If $\kappa_s >0$ in  \eqref{E:normal_r} for some $s$, and $\gamma$ of the form \eqref{E:parameter_gamma2} satisfies $\nu(r\circ\gamma) > 4M$, then

$$c_i^s=0\quad\text{ for }\quad i=0,\dots ,M-1.$$
\end{lemma}

\begin{proof}
Suppose not. Let $m_0$, $0\leq m_0\leq M-1$, be the smallest such integer with $c^s_{m_0}\neq 0$. Elementary expansion shows

$$\left|\gamma^s(t)\right|^2 =\left|c^s_{m_0}\right|^2\, \left|t\right|^{2(M+m_0)}+\co\left(|t|^{2(M+m_0)+1}\right).$$
Thus, in the coordinate direction $z_s$, i.e. for the curve $\sigma (t)= \left(0,\dots , \gamma^s(t),\dots, 0\right)$, it holds that

\begin{align*}
P_2\left(\sigma(t), \bar\sigma(t)\right) &= \kappa_s \left|c^s_{m_0}\right|^2\, \left|t\right|^{2(M+m_0)}+\co\left(|t|^{2(M+m_0)+1}\right) \\
&\geq c \left|t\right|^{2(M+m_0)}
\end{align*}
for $c>0$ and $t\in\C$ near the origin.

However the other $\kappa_j$ in \eqref{E:normal_r} are $\geq 0$, because of pseudoconvexity. Thus the terms $\kappa_j\left|\gamma^j(t)\right|^2$ that arise cannot cancel
the $ \kappa_s \left|c^s_{m_0}\right|^2\, \left|t\right|^{2(M+m_0)}$ term above. It follows that

\begin{align*}
P_2\left(\gamma(t), \bar\gamma(t)\right) \geq c \left|t\right|^{2(M+m_0)}.
\end{align*}
The higher-order terms in \eqref{E:normal_r} cannot interfere. Any terms coming from $P_3(z',\bar z')$ will have holomorphic degree $\geq 2M$ in $t$ or anti-holomorphic degree $\geq 2M$ in $\bar t$. Terms appearing because of degree $\geq4$ terms in $r$ will have degree $\geq 4M$. Lastly, terms involving the tangential $z_1$ direction need not be considered as $\gamma^1\equiv0$.
Consequently, 
\begin{equation*}
\left| r\circ\gamma(t)\right|\geq  c \left|t\right|^{2(M+m_0)}
\end{equation*}
as well. Since $m_0\leq M-1$, this says $\nu(r\circ\gamma)\leq 4M-2$, contradicting the assumption that $\nu(r\circ\gamma) > 4M$.
\end{proof}

When $\ch$ is pseudoconvex, it follows that $B(r;\gamma)$ in \eqref{E:(3M,M)_para} vanishes:

\begin{proposition}\label{P:last}
Let $\ch=\{r=0\}\subset\C^n$ be pseudoconvex smooth real hypersurface near $0\in\ch$, $r$ of the form \eqref{E:normal_r}, $\gamma$ of the form \eqref{E:parameter_gamma2} satisfying $\nu(r\circ\gamma)(0) > 4M$, and $B(r;\gamma)$ defined 
in \eqref{E:(3M,M)_para}.

Then $B(r;\gamma)=0$.
\end{proposition}

\begin{proof}
Suppose that $B(r;\gamma)\neq 0$. This implies, in particular,  that
there exist indices $j_0, k_0, \ell_0$ -- not necessarily distinct -- and an integer $0 <u<M$ such that

\begin{equation*}
r_{z_{j_0} z_{k_0} \bar z_{\ell_0}} (0)\, c_u^{j_0}\, c_{M-u}^{k_0}\, \bar c_0^{\ell_0} \neq 0.
\end{equation*}
This means that 

\begin{equation}\label{E:Bneq0}
\lambda_{j_0 k_0 \ell_0}\neq 0
\end{equation}
in $P_3(z,\bar z)$, as well as

\begin{align}\label{E:special_coeff}
c_u^{j_0}\neq 0,\quad c_{M-u}^{k_0}\neq 0,\quad c_0^{\ell_0}\neq 0.
\end{align}

However, \eqref{E:Bneq0} and Lemma \ref{L:1} force at least one of $\kappa_{j_0}$, $\kappa_{k_0}$, or $\kappa_{\ell_0}$ to be positive. For specificity, say $\kappa_{\ell_0}>0$.
Lemma \ref{L:2} then forces $c_0^{\ell_0}=0$, contradicting \eqref{E:special_coeff}. Thus $B(r;\gamma)$ must be 0, completing the proof.
\end{proof}

\begin{remark}
The proof of Proposition \ref{P:last} actually shows that each term in $B(r;\gamma)$ vanishes separately.
\end{remark}

Proposition \ref{P:last} and Lemma \ref{L:B} finish the proof for the case $\Delta^{reg}_1(0)=4$. Combining this case with Proposition \ref{P:low_type}, the proof of Theorem \ref{T:main} is complete.

\bibliographystyle{acm}
\bibliography{McnMer16}

\end{document}